\documentclass[12pt]{iopart}

\usepackage{graphicx,setspace}
\usepackage{psfrag}
\usepackage{amsfonts}
\expandafter\let\csname equation*\endcsname\relax
\expandafter\let\csname endequation*\endcsname\relax
\usepackage{amsmath}
\usepackage{amssymb, amsthm}
\usepackage{mathrsfs}
\usepackage{epstopdf}
\usepackage{float}
\usepackage{lineno,hyperref}
\usepackage{color}
\usepackage{subfigure}

\usepackage{amsmath}
\usepackage{dsfont} 
\usepackage{amssymb} 
\usepackage{graphicx,color}
\usepackage{extarrows}
\usepackage{CJK}
\usepackage{graphicx}

\usepackage{epstopdf}

\usepackage{amsfonts}
\usepackage{mathrsfs} 
\usepackage{caption}
\usepackage{subfigure}

\modulolinenumbers[5]

\begin{document}
\newtheorem{myexample}{Example}
\newtheorem{myDef}{Definition}
\newtheorem{myTheo}{Theorem}
\newtheorem{myLem}{Lemma}
\newtheorem{myCor}{Corollary}
\newtheorem{myCase}{Case}
\newtheorem*{mypro}{Proof}
\newtheorem{myrem}{Remark}
\newtheorem{definition}{Definition}
\newtheorem{lemma}{Lemma}
\newtheorem{remark}{Remark}
\newtheorem{theorem}{Theorem}
\newtheorem{proposition}{Proposition}
\newtheorem{assumption}{Assumption}
\newtheorem{example}{Example}
\newtheorem{corollary}{Corollary}
\def\ep{\varepsilon}
\def\Rn{\mathbb{R}^{n}}
\def\Rm{\mathbb{R}^{m}}
\def\E{\mathbb{E}}
\def\hte{\hat\theta}

\renewcommand{\theequation}{\thesection.\arabic{equation}}

\title[Characterization of the Most Probable Transition Paths] {Characterization of the Most Probable Transition Paths of Stochastic Dynamical Systems with Stable L\'{e}vy Noise}

\author{Yuanfei Huang$^1$, Ying Chao$^1$, Shenglan Yuan$^1$, Jinqiao Duan$^{2,*}$}

\address{$^1$School of Mathematics and Statistics \& Center for Mathematical Sciences \& Hubei Key Laboratory of Engineering Modeling and Scientific Computing, Huazhong University of Science and Technology, Wuhan 430074,  China}
\address{$^2$Department of Applied Mathematics, Illinois Institute of Technology, Chicago, IL 60616, USA}
\address{$^*$Author to whom correspondence should be addressed: duan@iit.edu}

\ead{yfhuang@hust.edu.cn, yingchao1993@hust.edu.cn, shanglanyuan@hust.edu.cn, duan@iit.edu.}

\begin{abstract}
This work is devoted to the investigation of the most probable transition path for stochastic dynamical systems driven by either symmetric $\alpha$-stable L\'{e}vy motion  ($0<\alpha<1$) or Brownian motion. For stochastic dynamical systems with Brownian motion, minimizing an action functional is a general method to determine the most probable transition path.  We have developed a method based on path integrals to obtain the most probable transition path of stochastic dynamical systems with symmetric $\alpha$-stable L\'{e}vy motion or Brownian motion, and the most probable path can be characterized by a deterministic dynamical system.  
\end{abstract}

\section{Introduction}

Stochastic differential equations (SDEs) have been widely used to describe  complex   phenomena in physical, biological, and  engineering systems. Transition phenomena between dynamically significant states occur in nonlinear systems under random fluctuations. Hence a practical problem is that given a stochastic dynamical system, how to capture the transition behavior between two metastable states,  and then how to determine the most probable transition path. This subject has been  the research topic  by a number of authors
$\cite{Durr1978}$-$\cite{Khandekar2000}$.

In this paper, we consider the following SDE in the state space $\mathbb{R}^k$:
\begin{equation}\label{firstequation}
dX_t=b(X_t)dt+dL_t,  \;\; ~T_0\leq t\leq T_f,  \;\; ~ X_{T_0}=X_{0},
\end{equation}
where $L_t$ is a $k$-dimensional symmetric $\alpha$-stable  (non-Gaussian) L\'{e}vy motion    in the probability space $(\Omega,   \mathbb{P})$. The solution process  $X_t$  uniquely exists under approppriate conditions on the drift term $b(x): \mathbb{R}^k\rightarrow \mathbb{R}^k$  (see the next section). Moreover, $T_0,~T_f$ are the initial and final time instants, respectively.
For simplicity,  we first consider  the one dimensional case  ($k=1$) and will extend to higher dimensional cases ($k>1$)  in Section 3.

We will also compare with the following  one dimensional SDE system  with (Gaussian) Browmnian motion $B_t$:
\begin{equation} \label{Btequation}
dX_t=b(X_t)dt+dB_t, \;\;     ~T_0\leq t\leq T_f,  \;\;  ~ X_{T_0}=X_{0}.
\end{equation}
The Onsager-Machlup function $\cite{Durr1978}$ and $Path$ $integrals$ $\cite{Kath1981,Wio2013,Tang2014}$ are two methods to study the most probable transition path of this system (\ref{Btequation}). The central points of these two methods were to express the transition probability (density) function of a diffusion process by means of a functional integral over paths of the process. That is for the solution process $X_t$,  with initial time and position  $(T_0,X_{0})$ and final time and   position  $(T_f,X_{f})$. In Stratonovich discretization prescription, the transition probability density $p(X_{f},T_f|X_{0},T_0)$ (or denoted by $p(X_{f}|X_{0})$) is expressed as a path integral
\begin{align}\label{pathintegral}
p(X_{f},T_f|X_{0},T_0)=\int_{\mbox{\tiny$\begin{array}{c}
x_{T_0}=X_{0}\\
x_{T_f}=X_{f}
\end{array}$}}\mathcal{D}x~\exp\{-\frac{1}{2}\int_{T_0}^{T_f}[(\dot{x_s}-b(x_s))^2+b'(x_s)]ds\},
\end{align}
where $\mathcal{L}(x,\dot{x})=\frac{1}{2}[(\dot{x}-b(x))^2+b'(x)]$ is called the Lagrangian of ($\ref{firstequation}$). In Onsager-Machlup's method, $OM(x,\dot{x})=(\dot{x}-b(x))^2+b'(x)$ is called the $OM$ function.   When the path $x_t$  is restricted in continuous functions mapping from $[T_0,T_f]$ into $\mathbb{R}$, the   exponent   $S(x)=-\frac{1}{2}\int_{T_0}^{T_f}[(\dot{x}-b(x))^2+b'(x)]ds$ is called the  Onsager-Machlup action functional. Hence finding the most probable transition path is to find a path $x_t$ such that the Lagrangian (OM function or the action functional) to be minimum, which is called  the least action principle. This leads to the Euler-Lagrangian equation by means of a variational principle when the path restricted in twice differentiable functions. For more details of $Path$ $integrals$ and   applications,  see
$\cite{Wio2013}$,$\cite{Fey1965}$-$\cite{Wio1999}$ and references therein.

In this present paper, we will determine the most probable transition path  for the stochastic system with non-Gaussian noise (\ref{firstequation}).  The situation is different from the Gaussian case
(\ref{Btequation}).     If we try to get     the exponential form (containing the action functional) for the transition probability density function for the transition path  as in the Gaussian case,  we    need to use the Fourier transformation of the probability density $\cite{laskin2002,Jana2012}$ (or   characteristic function). For instance, the characteristic function of a $\alpha$-stable L\'{e}vy random variable is $\cite{Applebaum,Sato1999}$
\begin{align}
\psi(u)=\exp\{i\eta u-\sigma^{\alpha}|u|^{\alpha}[1-i\beta\frac{u}{|u|}w(u,\alpha)]\},
\end{align}
where $0<\alpha<2$ is the L\'{e}vy index, $-1\leq\beta\leq1$ is the skewness parameter, $\eta\in \mathbb{R}$ is the shift parameter, $\sigma\in \mathbb{R}+$   the scale parameter and
\begin{equation}
 w(u,\alpha)= \begin{cases}
 \tan(\frac{\pi \alpha}{2}),~\alpha\neq1,\\
 -\frac{2}{\pi}\ln(|u|),~\alpha=1.\\
  \end{cases}
\end{equation}
The density function of this random variable is
\begin{align}
f(x)=\frac{1}{2\pi}\int_{-\infty}^{+\infty}\exp(-iux)\psi(u)du.
\end{align}

Thus it brings the Fourier integral into the density function. So for path integral representation with this density form, it is hard  (and this is unlike $(\ref{pathintegral})$)  to obtain a convergent action function representation of paths:
\small{\begin{align}
p(X_{f},T_f|X_{0},T_0)=\lim_{n\rightarrow\infty}\int_{-\infty}^{\infty} \cdots\int_{-\infty}^{\infty} \Pi_{n}f(x_n-x_{n-1}-b(x_{n-1}\triangle t))\delta(x_{n}-X_{f})\mathcal{J}dx_1\cdots dx_{n},
\end{align}}
where $n$ is the partition number and $\mathcal{J}$ is the Jacobian of the   transformation given by
\begin{equation}
\begin{split}
\mathcal{J}=\det(\frac{\partial L_i}{\partial x_k}),~i=1,\cdots,n,~k=1,\cdots,n.
\end{split}
\end{equation}

Instead,  in this paper, we develop a method to characterize the most probable transition path,   based on the path integral rather than on the action functional (or the Onsager-Machlup function). This is made possible with a new representation   $\cite{Niels2012}$ for the transition probability density functions of symmetric L\'{e}vy motions  in terms of two families of metrics. This representation provides  an exponential structure of  the transition probability density function  $\cite{Niels2012,Jacob2001}$. It can be further extended   in our case,  which will be discussed in Section 2.2.

This paper is organized as follows. In Section 2, we recall some preliminaries. In Section 3, we develop a method to characterize the most probable transition paths for  a stochastic system  with symmetric $\alpha$-stable L\'{e}vy motion  ($0<\alpha<1$) or Brownian motion. In Section 4, we extend the results of Section 3 to higher dimensional cases.  Finally,  in Section 5, we present several examples to illustrate our results.

\section{Preliminaries}
\subsection{L\'{e}vy motions}
We recall some basic facts  about  1-dimensional L\'{e}vy motions (or L\'{e}vy processes)  \cite{Applebaum,Sato1999,Duan2015}.

\begin{myDef}
A stochastic process $L_t$ is a L\'{e}vy process if\\
(\romannumeral1) $L_0$=0 (a.s.);\\
(\romannumeral2) $L_t$ has independent increments and stationary increments; and\\
(\romannumeral3) $L_t$ has stochastically continuous sample paths, $i.e.$, for every $s\geq0$, $L(t)\rightarrow L(s)$ in probability, as $t\rightarrow s$.
\end{myDef}

A L\'{e}vy process $L_t$ taking values in $\mathbb{R}$ is characterized by a drift term $\eta\in \mathbb{R}$, a non-negative variance $Q$ and a Borel measure $\nu$ defined on $\mathbb{R}\setminus \{0\}$. $(\eta,Q,\nu)$ is called the generating triplet of the L\'{e}vy motion $L_t$. Moreover, the L\'{e}vy-It\^{o} decomposition for $L_t$ as follows:
\begin{align}
L_t=\eta t+B_{Q}(t)+\int_{|z|<1}z\tilde{N}(t,dz)+\int_{|z|\geq1}zN(t,dz),
\end{align}
where $N(dt,dz)$ is the Poisson random measure, $\tilde{N}(dt,dz)=N(dt,dz)-\nu(dz)dt$ is the compensated Poisson random measure, and $\nu(S)\triangleq \mathbb{E}N(1,S)$, here $\mathbb{E}$  denotes the expectation with respect to the probability $\mathbb{P}$, and $B_{\sigma}(t)$ is a Brownian  motion with variance $\sigma$. The characteristic function of $L_t$ is given by
\begin{align}
\mathbb{E}[\exp(i u L_t )]=\exp(-t \psi(u)),~~u\in \mathbb{R},
\end{align}
where the function $\psi: \mathbb{R}\rightarrow \mathbb{C}$ is the characteristic exponent
\begin{align}
\psi(u)=i u \eta +\frac{1}{2} Q u^2 +\int_{\mathbb{R}\setminus\{0\}}(1- e^{i u z }+i u z  \mathbb{I}_{\{|z|<1\}})\nu(dz).
\end{align}
The Borel measure $\nu$ is called the jump measure.

\subsection{Asymptotic properties of the probability density functions of $\alpha$-stable L\'{e}vy motions}
From now on, we consider a scalar symmetric $\alpha$-stable L\'{e}vy processes.  Recall the standard symmetric $\alpha$-stable random variable has  distribution  $S_{\alpha}(1,0,0)$. Here $S_{\alpha}(\sigma,\beta,\mu)$ is the distribution of a stable random variable, with  $\sigma$   the scale parameter, $\beta$   the skewness parameter and $\mu$   the shift parameter.   The corresponding probability density function $f_{\alpha}(x)$  can be represented as an  infinite series $\cite{Duan2015,Janicki1994,Shao1993}$
\begin{equation}
  f_{\alpha}(x)=\begin{cases}
 \frac{1}{\pi \alpha}\sum_{k=1}^{\infty}\frac{(-1)^{k-1}}{k!}\Gamma(\alpha k+1)|x|^{-\alpha k}\sin(\frac{\alpha k \pi}{2}),&x\neq0,~0<\alpha<1,\\
 \frac{1}{\pi}\int_{0}^{\infty}e^{-u^{\alpha}}du,&x=0,~0<\alpha<1,\\
 \frac{1}{\pi (1+x^2)},&\alpha=1,\\
 \frac{1}{\pi \alpha}\sum_{k=0}^{\infty}\frac{(-1)^{k}}{(2k)!}\Gamma(\frac{2k+1}{\alpha})x^{2k},&1<\alpha<2.
  \end{cases}
\end{equation}
Recall the probability density function $f_{b}(x)$ for a Brownian random variable $X\sim N(0,\sigma^2)$ is $\cite{Duan2015}$
\begin{equation}
  f_{b}(x)=\frac{1}{\sqrt{2\pi} \sigma}e^{-\frac{x^2}{2\sigma^2}}.\\
\end{equation}

In $\cite{Jana2012}$, it was proved that the transition probability density function $p(x_t|x_0)$ of a symmetric $\alpha$-stable L\'{e}vy process is
\begin{equation}
\begin{split}
p(x_t|x_0)&=\frac{1}{t^{1/\alpha}}f_{\alpha}(\frac{x_t-x_0}{t^{1/\alpha}})\\
&=\frac{1}{t^{1/\alpha}}\exp[-(-\ln f_{\alpha}(\frac{x_t-x_0}{t^{1/\alpha}}))]\\
&\triangleq \frac{1}{t^{1/\alpha}}\exp[-\theta^{\alpha}_t(x_t-x_0)],
\end{split}
\end{equation}
where $\theta^{\alpha}_t(\cdot)$ is a function maps $[0,\infty)$ to $[0,\infty)$ for any $\alpha\in(0,2)$ and $t\in(0,\infty)$. Differentiate $\theta^{\alpha}_t(x)$  with respect to variable $x$:
\begin{equation}
\begin{split}
(\theta^{\alpha}_t(x))'&=(-\ln f_{\alpha}(\frac{x}{t^{1/\alpha}}))'\\
&=-\frac{f'_{\alpha}(\frac{x}{t^{1/\alpha}})}{f_{\alpha}(\frac{x}{t^{1/\alpha}})}\frac{1}{t^{1/\alpha}},
\end{split}
\end{equation}
which shows that $\theta^{\alpha}_t(\cdot)$ is a strict increase function since $f'_{\alpha}(\frac{x}{t^{1/\alpha}})\leq0$ for symmetric $\alpha$-stable L\'{e}vy random variables. Now we focus on the concavity of the function $\theta^{\alpha}_t(\cdot)$.
\begin{equation}
\begin{split}
(\theta^{\alpha}_t(x))''&=(-\frac{f'_{\alpha}(\frac{x}{t^{1/\alpha}})}{f_{\alpha}(\frac{x}{t^{1/\alpha}})}\frac{1}{t^{1/\alpha}})'\\
&=-\frac{f_{\alpha}(\frac{x}{t^{1/\alpha}})f''_{\alpha}(\frac{x}{t^{1/\alpha}})-(f'_{\alpha}(\frac{x}{t^{1/\alpha}}))^2}{f^2_{\alpha}(\frac{x}{t^{1/\alpha}})}\frac{1}{t^{2/\alpha}},
\end{split}
\end{equation}

In $\cite{Applebaum,Janicki1994}$, it was proved that if $X\sim S_{\alpha}(\sigma,\beta,\mu)$ and $\alpha\in(0,2)$,
\begin{equation}
\begin{split}
\lim_{y\rightarrow\infty}y^{\alpha}\mathbb{P}(X>y)=C_{\alpha}\frac{1+\beta}{2}\sigma^{\alpha},\\
\lim_{y\rightarrow\infty}y^{\alpha}\mathbb{P}(X<-y)=C_{\alpha}\frac{1-\beta}{2}\sigma^{\alpha},
\end{split}
\end{equation}
where
\begin{align}
C_{\alpha}=(\int_{0}^{\infty}x^{-\alpha}\sin(x)dx)^{-1}.
\end{align}

We use this result to study the asymptotic behavior of tail probabilities. For $y$ large enough,
\begin{equation}\label{asymptotic}
\begin{split}
&y^{\alpha}\mathbb{P}(X>y)=C_{\alpha}\frac{1+\beta}{2}\sigma^{\alpha},\\
\Leftrightarrow~&y^{\alpha}\int_{y}^{\infty}f_{\alpha}(x)dx=C_{\alpha}\frac{1+\beta}{2}\sigma^{\alpha},\\
\Leftrightarrow~&f_{\alpha}(y)=\alpha C_{\alpha}\frac{1+\beta}{2}\sigma^{\alpha} y^{-\alpha-1}\triangleq Cy^{-\alpha-1},\\
\Leftrightarrow~&f''_{\alpha}(y)f_{\alpha}(y)-(f'_{\alpha}(y))^2=C^2(\alpha+1)y^{-2\alpha-4},
\end{split}
\end{equation}
which means that the asymptotic behavior of tail concavity and convexity of $\theta^{\alpha}_t(\cdot)$ is concave. And the graphs of $\theta^{\alpha}_1(\cdot)$ are shown in Figure $\ref{concavefigure}$.

For Brownian motion case, similar to the symmetric $\alpha$-stable L\'{e}vy motion case, the corresponding exponent is $\theta^b_{t}(x)=\frac{x^2}{2t}$, which is convex.

We also notice that
\begin{equation}
\begin{split}
&\theta^{\alpha}_t(x-y)=\theta^{\alpha}_t(|x-y|)\triangleq\theta^{\alpha}_t(x,y),\\
&\theta^{b}_t(x-y)=\theta^{b}_t(|x-y|)\triangleq\theta^{b}_t(x,y).
\end{split}
\end{equation}
\begin{figure}
 \centering
 \subfigure[]{
 \includegraphics[width=7cm,height=6cm]{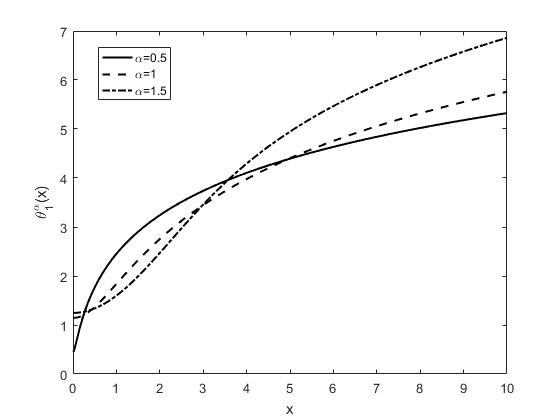}
 }
 \subfigure[]{
 \includegraphics[width=7cm,height=6cm]{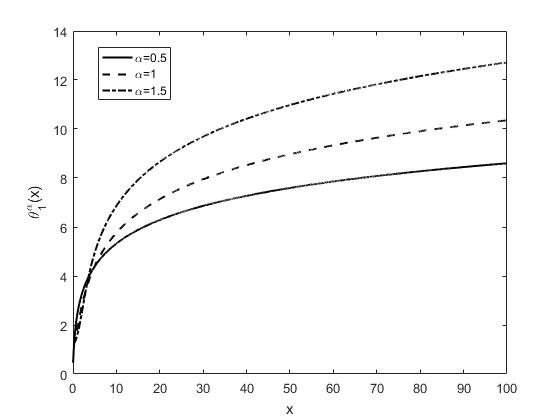}
 }
 \caption{The graphs of function $\theta^{\alpha}_1(x)$ (with $\alpha=0.5,~1~$ and~$1.5$):  (a)   on $[0,10]$;  (b)  on $[0,100]$.}\label{concavefigure}
 \end{figure}

\subsection{Conditions for the well-posedness of the system ($\ref{firstequation}$) and ($\ref{Btequation}$)}

For the system ($\ref{firstequation}$):
\begin{equation*}
dX_t=b(X_t)dt+dL_t,~T_0\leq t\leq T_f,~ X_{T_0}=X_{0},
\end{equation*}
it was proved in $\cite{Albeverio2010,Xi2018}$ that if $b(x)$ is locally Lipschitz continuous function and satisfies ``one sided linear growth'' condition in the following sense:
\begin{itemize}
\item {\bf C1 (Locally Lipschitz condition)}  For any $r>0$, there exists $K_1>0$ such that, for all
$|y_1|,|y_2|\leq r$,
$$|b(y_1)-b(y_2)|^2\leq K_1|y_1-y_2|^2,$$
\item {\bf C2 (One sided linear growth condition)} There exists $K_2>0$ such that, for all $y\in \mathbb{R}$,
$$2y\cdot b(y)\leq K_2(1+|y^2|),$$
\end{itemize}
then there exists a unique global solution to ($\ref{firstequation}$) and the solution is adapted and c\`{a}dl\`{a}g. These two conditions also guarantee  the existence and uniqueness for the solution of ($\ref{Btequation}$).

\section{Method}
In this section, we first study the transition behavior of the system without drift term, and  then we study the general case for system ($\ref{firstequation}$) and ($\ref{Btequation}$). Before we develop the method,  we present the following definition.
\begin{myDef}
For a solution  path $X_t$ and the time interval $[T_0,T_f]$ with a partition $T_0=t_0\leq t_1\leq\cdots\leq t_n=T_f$, define the sequence $\{X_{t_0},X_{t_1},\cdots,X_{t_n}\}$ as a discretized path of $X_t$. And a discretized path $\{X_{t_0},X_{t_1},\cdots,X_{t_n}\}$ is said to be monotonic
 with respect to the time partition $T_0=t_0\leq t_1\leq\cdots\leq t_n=T_f$ if either  $X_{t_0}\leq X_{t_1}\leq\cdots\leq X_{t_n}$ or $X_{t_0}\geq X_{t_1}\geq\cdots\geq X_{t_n}$
\end{myDef}

\subsection{Most probable transition path in the absence of drift}
\begin{myTheo}\label{mtheorem}
({Monotonicity for the most probable transition path  in the absence of drift}) \\
For   system ($\ref{firstequation}$) and system $\ref{Btequation}$ with drift $b\equiv0$,   every discretized path $\{X_{t_0},X_{t_1},\cdots,X_{t_n}\}$ of the most probable transition path  is monotonic
 with respect to every time partition $T_0=t_0\leq t_1\leq\cdots\leq t_n=T_f$.
\end{myTheo}
\begin{proof}
For a time interval partition $T_0=t_0\leq t_1\leq t_2\leq \cdots\leq t_n=T_f$,  define $t_{i+1}-t_i \triangleq \epsilon=(T_f-T_0)/n,~i=0,1,\cdots,n-1$.  (This proof is also true for an arbitrary partition.) In path integral method,  the transition density function (or   Markov transition probability) of $X_t$ of ($\ref{firstequation}$) with $b\equiv0$ is
\begin{equation}\label{pdf}
\begin{split}
p(X_{f},T_f|X_{0},T_0)&=\int_{-\infty}^{\infty}\cdots\int_{-\infty}^{\infty}dx_{t_1}\cdots dx_{{t_{n-1}}}p_{\epsilon}(x_{t_1}|X_{0})p_{\epsilon}(x_{t_2}|x_{t_1})\cdots p_{\epsilon}(X_{f}|x_{{t_{n-1}}})\\
&=\int \mathcal{D}_n x ~\exp\{-\theta^{\alpha}_{\epsilon}(x_{{t_1}},X_{0})-\sum_{i=1}^{n-2}\theta^{\alpha}_{\epsilon}(x_{{t_{i+1}}},x_{t_i})-\theta^{\alpha}_{\epsilon}(X_{f},x_{{t_{n-1}}})\}\\
&\triangleq \int \mathcal{D}_n x ~\exp\{-\sum_{i=0}^{n-1}\theta^{\alpha}_{\epsilon}(x_{{t_{i+1}}},x_{t_i})\}\\
&\triangleq \int \mathcal{D}_n x ~\exp\{-\mathcal{S}_n(x)\},
\end{split}
\end{equation}
where $\mathcal{D}_nx=\epsilon^{-1/\alpha}\prod_{i=1}^{n-1}\epsilon^{-1/\alpha}dx_{t_i}$.  Note that   $x$ of $\mathcal{S}_n(x)$ is a path connecting  $(T_0,X_{0})$ and $(T_f,X_{f})$ ($i.e.$ it starts at $X_0$ at the time $t=T_0$ and reaches $X_f$ at the time $t=T_f$).   We call $\mathcal{S}_n(x)$   the \textbf{action quantity} of $x$. The contribution of a path $x_t$ to the transition probability density $p_t(X_{f},T_f|X_{0},T_0)$ depends on $\mathcal{S}_n(x)=\sum_{i=0}^{n-1}\theta^{\alpha}_{\epsilon}(x_{{t_{i+1}}},x_{t_i})$.

In order to find the most probable transition path $u_t$,   we are supposed to find it satisfies that
\begin{align}\label{s1}
\mathcal{S}_n(u_t)=\min_{x_t\in D_{X_{0}}^{X_{f}}}\mathcal{S}_n(x_t),
\end{align}
where $D_{X_{0}}^{X_{f}}$ denotes the set of paths that connect $(T_0,X_0)$ and $(T_f,X_f)$.    Equivalently
\begin{align}\label{s2}
\frac{\mathcal{S}_n(x_t)}{\mathcal{S}_n(u_t)}\geq 1,
\end{align}
for every path $x_t$ connecting $(T_0,X_{0})$ and $(T_f,X_{f})$. Here $n$ goes to infinity, and the $\lim_{n\rightarrow\infty}$   is dropped for now for clarity.

Without loss of generality, we set $X_{0}<X_{f}$. When $n=2$ (the time partition is $\{T_0=t_0<t_1<t_2=T_f\}$), for a path $x_t$ connecting $(T_0,X_0)$ and $(T_f,X_f)$, $\mathcal{S}_2(x_t)=\theta^{\alpha}_{\epsilon}(x_{t_1},X_{0})+\theta^{\alpha}_{\epsilon}(X_{f},x_{t_1})$, if $x_{t_1}<X_{0}$. Since $\theta^{\alpha}_{\epsilon}(\cdot)$ is a strictly increasing function,  together with $|x_{t_1}-X_{0}|>|X_0-X_{0}|=0$ and $|X_{f}-x_{t_1}|>|X_{f}-X_{0}|$, we have $\mathcal{S}_2(x_t)>\theta^{\alpha}_{\epsilon}(X_{0},X_{0})+\theta^{\alpha}_{\epsilon}(X_{0},X_{f})$ where $\{X_{0},x_{t_1}=X_{f},X_{f}\}$ is a discretized  path of a certain transition path connecting $(T_0,X_{0})$ and $(T_f,X_{f})$. The case that $x_{t_1}>X_{f}$ is similar. When we add time partitions the situation is similar again. Thus the discretized  path $\{X_{t_0},X_{t_1},\cdots,X_{t_n}\}$ of the most probable transition path of $X_t$ is supposed to be monotonic with respect to the time partition, otherwise there exists a path whose action  quantity is smaller.

For system ($\ref{Btequation}$), the proof is similar.
\end{proof}

The reason that we use ($\ref{s2}$) to find the most probable transition path is: For symmetric $\alpha$-stable L\'{e}vy motion case, the corresponding action quantity $\mathcal{S}_n(x_t)$ goes to infinity as long as $n$ goes to infinity for any path $x_t\in D_{X_0}^{X_f}$ (since $\mathcal{S}_n(x_t)\geq n\theta_{\epsilon}^{\alpha}(0)$ and $\theta_{\epsilon}^{\alpha}(0)$ is a positive constant. Here $n\theta_{\epsilon}^{\alpha}(0)$ is the action quantity of the path $X_t\equiv X_0$, $t\in[T_0,T_f]$). For simplicity we say the action quantity of $x_t$ has higher order than $u_t$ if
\begin{align}
\lim_{n\rightarrow\infty}\frac{\mathcal{S}_n(x_t)-n\theta_{\epsilon}^{\alpha}(0)}{\mathcal{S}_n(u_t)-n\theta_{\epsilon}^{\alpha}(0)}=+\infty.
\end{align}
That is, we compare the action quantities of two paths with the help of the fixed path ($X_t\equiv X_0$, $t\in[T_0,T_f]$), which considers $X_t\equiv X_0$ $(t\in[T_0,T_f])$ as a reference path. It helps us to compare the action quantities easily, which will be shown in the proof of Corollary $\ref{lemmaconvex}$. Actually,
\begin{align*}
&\mathcal{S}_n(u_t)-n\theta_{\epsilon}^{\alpha}(0)\\
=&\sum_{i=1}^{n}\theta_{\epsilon}^{\alpha}(x_i,x_{i-1})-n\theta_{\epsilon}^{\alpha}(0)\\
=&\sum_{i=1}^{n}[\theta_{\epsilon}^{\alpha}(x_i,x_{i-1})-\theta_{\epsilon}^{\alpha}(0)]\\
=&\sum_{i=1}^{n}[-\ln f_{\alpha}(x_i-x_{i-1})+\ln f_{\alpha}(0)]\\
=&\sum_{i=1}^{n}[-\ln \frac{f_{\alpha}(x_i-x_{i-1})}{f_{\alpha}(0)}].
\end{align*}
Denoting $g_{\alpha}(x_i-x_{i-1})=\frac{f_{\alpha}(x_i-x_{i-1})}{f_{\alpha}(0)}$, the function $g_{\alpha}(\cdot)$ could be regarded  as a new ``probability'' density function whose integration over $(-\infty,\infty)$ is $\frac{1}{f_{\alpha}(0)}$.

The most probable transition path might not be unique, $i.e.$, there might be several paths satisfying (\ref{s1}) or (\ref{s2}). In fact we will see in the following corollary that, for system (\ref{firstequation}), the number of  the most probable transition paths is infinity but they can be characterized by a class of paths that have only one jump, and the difference among the paths of this class is the jump time. See Remark \ref{remark1}  after the proof of the following corollary.

\begin{corollary}\label{lemmaconvex}
(\romannumeral1) For system ($\ref{firstequation}$) with $b\equiv0$,  if $L_t$ is a symmetric $\alpha$-stable L\'{e}vy noise with $0<\alpha<1$, then the most probable transition path is not unique, and it can be represented as a $Heaviside$-like function
\begin{equation}\label{mpptp}
  X^m_{t}=\begin{cases}
X_{0},~~T_0\leq t<t^*,\\
X_{f},~~t^*\leq t\leq T_f,\\
  \end{cases}
\end{equation}
for  every  time instant $t^*$ satisfying $T_0<t^* \leq T_f$;

(\romannumeral2) For system ($\ref{Btequation}$) with  drift $b\equiv0$, the most probable transition path is  the line segment connecting $X_0$ and $X_f$:    $X^m_{t}=X_{0}+\frac{t-T_0}{T_f-T_0}(X_{f}-X_{0})$;
\end{corollary}
\begin{proof}

(\romannumeral1) For system ($\ref{firstequation}$) with drift $b\equiv0$, as shown in Section 2, we notice that for a time interval partition $T_0=t_0\leq t_1\leq \cdots\leq t_n=T_{f}$, $t_{i+1}-t_i=\epsilon=\frac{T_f-T_0}{n}$,
\begin{align}
\theta^{\alpha}_{\epsilon}(|x-y|)=\theta^{\alpha}_{1}(|\frac{x-y}{\epsilon^{1/\alpha}}|)=\theta^{\alpha}_{1}(n^{\alpha}|\frac{x-y}{(T_f-T_0)^{1/\alpha}}|).
\end{align}

Define a path space  $DM_{X_{0}}^{X_{f}}=\{x_t|~x_t$ is a monotonic path connects $(T_0,X_{0})$  and $(T_f,X_{f})$\}. So one should search for the most probable transition path within path space $DM_{X_{0}}^{X_{f}}$.

Take a path $x_t\in DM_{X_{0}}^{X_{f}}$  and a time partition   $\{T_0=t_0<t_1<\cdots<t_n=T_f\}$. Assume that $\{\lambda_j\}_{j=0}^{n-1}$ are non-negative constants, and $x_{t_{j+1}}-x_{t_{j}}=\lambda_j(X_{f}-X_{0})$ $(0\leq j \leq n-1)$. It is easy to see that $\sum_{j=0}^{n-1}\lambda_j=1$ by the Theorem $\ref{mtheorem}$, and $0\leq\lambda_j\leq1$. As discussed in Section 2.2, $\theta^{\alpha}_1(\cdot)$ is concave in $[r,\infty)$ for some constant $r\in\mathbb{R}^+$ ($r$   depending on  $\alpha$).

For $n^{1/\alpha}\lambda_j|\frac{X_{f}-X_{0}}{(T_f-T_0)^{1/\alpha}}|\triangleq n^{1/\alpha}\lambda_jC_{0f}\geq r$, and $n$ large enough, we obtain
\begin{equation}\label{concavity}
\begin{split}
\theta^{\alpha}_{\epsilon}(|x_{t_{j+1}}-x_{t_{j}}|)-c&=\theta^{\alpha}_{1}(n^{1/\alpha}\lambda_j|\frac{X_{f}-X_{0}}{(T_f-T_0)^{1/\alpha}}|)-c\\
&\geq \lambda_j(\theta^{\alpha}_{1}(n^{1/\alpha}|\frac{X_{f}-X_{0}}{(T_f-T_0)^{1/\alpha}}|)-c)\\
&= \lambda_j (\theta^{\alpha}_{\epsilon}(|X_{f}-X_{0}|)-c),
\end{split}
\end{equation}
where $c=\theta^{\alpha}_{\epsilon}(0)$ is a positive constant and $\theta^{\alpha}_1(\cdot)-c$ is non-negative and concave in $[r,\infty)$.

When $n^{1/\alpha}\lambda_jC_{0f}<r$, we have
\begin{equation}
\begin{split}
0\leq\sum_{n^{1/\alpha}\lambda_jC_{0f}<r}\lambda_j<n\frac{r}{n^{1/\alpha}C_{0f}},
\end{split}
\end{equation}
and
\begin{equation}\label{control}
\begin{split}
\lim_{n\rightarrow\infty}\sum_{n^{1/\alpha}\lambda_jC_{0f}<r}\lambda_j=0,~\lim_{n\rightarrow\infty}\sum_{n^{1/\alpha}\lambda_jC_{0f}\geq r}\lambda_j=1.
\end{split}
\end{equation}

Hence
\small{\begin{equation}
\begin{split}
&\lim_{n\rightarrow\infty}\mathcal({S}_n(x_t)-n\theta_{\epsilon}^{\alpha}(0))/\{[(n-1)\theta_{\epsilon}^{\alpha}(0)+\theta^{\alpha}_{1}(n^{1/\alpha}C_{0f})]-n\theta_{\alpha}^{\epsilon}(0)\}\\
&=\lim_{n\rightarrow\infty}[\sum_{j=0}^{n-1}\theta^{\alpha}_{1}(n^{1/\alpha}\lambda_jC_{0f})-n\theta_{\epsilon}^{\alpha}(0)]/[\theta^{\alpha}_{1}(n^{1/\alpha}C_{0f})-\theta_{\epsilon}^{\alpha}(0)]\\
&=\lim_{n\rightarrow\infty}[\sum_{n^{1/\alpha}\lambda_jC_{0f}<r}(\theta^{\alpha}_{1}(n^{1/\alpha}\lambda_jC_{0f})-\theta_{\epsilon}^{\alpha}(0))+\sum_{n^{1/\alpha}\lambda_jC_{0f}\geq r}(\theta^{\alpha}_{1}(n^{1/\alpha}\lambda_jC_{0f})-\theta_{\epsilon}^{\alpha}(0))]/(\theta^{\alpha}_{1}(n^{1/\alpha}C_{0f})-\theta_{\epsilon}^{\alpha}(0))\\
&\geq \lim_{n\rightarrow\infty}\sum_{n^{1/\alpha}\lambda_jC_{0f}\geq r}(\theta^{\alpha}_{1}(n^{1/\alpha}\lambda_jC_{0f})-\theta_{\epsilon}^{\alpha}(0))/(\theta^{\alpha}_{1}(n^{1/\alpha}C_{0f})-\theta_{\epsilon}^{\alpha}(0))\\
&\geq\lim_{n\rightarrow\infty}\sum_{n^{1/\alpha}\lambda_jC_{0f}\geq r}\lambda_j(\theta^{\alpha}_{1}(n^{1/\alpha}C_{0f})-\theta_{\epsilon}^{\alpha}(0))/(\theta^{\alpha}_{1}(n^{1/\alpha}C_{0f})-\theta_{\epsilon}^{\alpha}(0))\\
&=\lim_{n\rightarrow\infty}\sum_{n^{1/\alpha}\lambda_jC_{0f}\geq r}\lambda_j\\
&=1.
\end{split}
\end{equation}}

This means that the most probable transition path for symmetric $\alpha$-stable L\'{e}vy process ($0<\alpha<1$) is a $Heaviside$-like function
\begin{equation}
 X^m_{t}\triangleq\begin{cases}
 X_{f},~~t\geq t^*,\\
 X_{0},~~t<t^*,\\
  \end{cases}
\end{equation}
or
\begin{equation}
 X^m_{t}\triangleq\begin{cases}
 X_{f},~~t> t^*,\\
 X_{0},~~t\leq t^*,\\
  \end{cases}
\end{equation}
where $t^*$  satisfies $T_0\leq t^* \leq T_f$. Hence in this case, the most probable transition path is not unique since the ``jump time'' $t^*$ can be chosen arbitrarily.

(\romannumeral2) For system ($\ref{Btequation}$) with $b\equiv0$, by Theorem $\ref{mtheorem}$  and  the fact that $\theta^b_{\epsilon}(\cdot)$ being  convex, we conclude that
\begin{equation}\label{browniancovex}
\begin{split}
\sum_{i=0}^{n-1}\theta^b_{\epsilon}(x_{t_{i+1}},x_{t_i})&=\sum_{i=0}^{n-1}\theta^b_{\epsilon}(|x_{{t_{i+1}}}-x_{t_i}|)\\
&\geq n\theta^b_{\epsilon}(\sum_{i=0}^{n-1}\frac{1}{n}|x_{{t_{i+1}}}-x_{t_i}|)\\
&=n\theta^b_{\epsilon}(\frac{1}{n}|X_{f}-X_{0}|).
\end{split}
\end{equation}
The inequality holds if and only if
\begin{equation}
\begin{split}
\theta^b_{\epsilon}(x_{t_1},x_{t_0})=\theta^b_{\epsilon}(x_{t_2},x_{t_1})=\cdots=\theta^b_{\epsilon}(x_{{t_n}},x_{{t_{n-1}}})\\
\Leftrightarrow x_{t_1}-x_{t_0} = x_{t_2}-x_{t_1} =\cdots= x_{{t_n}}-x_{{t_{n-1}}}.
\end{split}
\end{equation}
So the most probable transition path is the line segment  which connects the initial and final points. Therefore,  we obtain
\begin{align}
X^m_{t_i}=\frac{i(X_{f}-X_{0})}{n},~i=1,2,\cdots,n-1.
\end{align}
When $n$ goes to infinity, $X^m_{t}=X_{0}+\frac{t-T_0}{T_f-T_0}(X_{f}-X_{0})$. This implies that the  most probable transition path is the path for the particle (i.e., solution) moving in constant velocity.
\end{proof}

\begin{myrem}\label{remark1}
For symmetric $\alpha$-stable L\'{e}vy motion with $0<\alpha<1$, the proof of Corollary $\ref{lemmaconvex}$ compares the transition paths' action quantities in path-wise sense.  We now study the probability over all paths starting at $X_0$ at time $T_0$ and conditioned at a given end point $X_f$ at time $T_f$, to find the particle at point $X$ at time $t\in[T_0,T_f]$. This probability can be written as (without loss of generality we assume $X_f>X_0$)
\begin{equation}
\begin{split}
\mathcal{P}(X,t)&=\frac{p(X,t|X_0,T_0)p(X_f,T_f|X,t)}{p(X_f,T_f|X_0,T_0)}\\
&=\frac{1}{p(X_f,T_f|X_0,T_0)}\frac{1}{|t-T_0|^{1/\alpha}|T_f-t|^{1/\alpha}}f_{\alpha}(\frac{X-X_0}{|t-T_0|^{1/\alpha}})f_{\alpha}(\frac{X_f-X}{|T_f-t|^{1/\alpha}}),
\end{split}
\end{equation}
which was studied by $\cite{Delarue2017,zheng2017}$. So when $t$ is fixed, the probability $\mathcal{P}(X,t)$ is a function depending on $f_{\alpha}(\frac{X-X_0}{|t-T_0|^{1/\alpha}})f_{\alpha}(\frac{X_f-X}{|T_f-t|^{1/\alpha}})$.  Note  that $f_{\alpha}(\frac{X-X_0}{|t-T_0|^{1/\alpha}})$ has a peak at $X_0$, and $f_{\alpha}(\frac{X_f-X}{|T_f-t|^{1/\alpha}})$ has a peak at $X_f$. Thus the product $f_{\alpha}(\frac{X-X_0}{|t-T_0|^{1/\alpha}})f_{\alpha}(\frac{X_f-X}{|T_f-t|^{1/\alpha}})$ increases as $X\uparrow X_0$ and decreases as $X\downarrow X_f$. That is, the product reaches the global maximal value in $[X_0,X_f]$. Suppose that $X_0\leq X\leq X_f$. The product can be rewritten as
\begin{equation}
\begin{split}
f_{\alpha}(\frac{X-X_0}{|t-T_0|^{1/\alpha}})f_{\alpha}(\frac{X_f-X}{|T_f-t|^{1/\alpha}})=\exp(-(\theta^{\alpha}_{t-T_0}(X-X_0)+\theta^{\alpha}_{T_f-t}(X_f-X))).
\end{split}
\end{equation}
Hence
\begin{equation}
\begin{split}
&\theta^{\alpha}_{t-T_0}(X-X_0)+\theta^{\alpha}_{T_f-t}(X_f-X)\\
=&\theta^{\alpha}_{1}(\frac{X_f-X_0}{|t-T_0|^{1/\alpha}}\frac{X-X_0}{X_f-X_0})+\theta^{\alpha}_{1}(\frac{X_f-X_0}{|T_f-t|^{1/\alpha}}\frac{X_f-X}{X_f-X_0})\\
\geq&\frac{X-X_0}{X_f-X_0}\theta^{\alpha}_{1}(\frac{X_f-X_0}{|t-T_0|^{1/\alpha}})+\frac{X_f-X}{X_f-X_0}\theta^{\alpha}_{1}(0)+\frac{X_f-X}{X_f-X_0}\theta^{\alpha}_{1}(\frac{X_f-X_0}{|T_f-t|^{1/\alpha}})+\frac{X-X_0}{X_f-X_0}\theta^{\alpha}_{1}(0)\\
\geq&(\frac{X-X_0}{X_f-X_0}+\frac{X_f-X}{X_f-X_0})\min\{\theta^{\alpha}_{1}(\frac{X_f-X_0}{|t-T_0|^{1/\alpha}}),\theta^{\alpha}_{1}(\frac{X_f-X_0}{|T_f-t|^{1/\alpha}})\}+\theta^{\alpha}_{1}(0)\\
=&\min\{\theta^{\alpha}_{1}(\frac{X_f-X_0}{|t-T_0|^{1/\alpha}}),\theta^{\alpha}_{1}(\frac{X_f-X_0}{|T_f-t|^{1/\alpha}})\}+\theta^{\alpha}_{1}(0),
\end{split}
\end{equation}
where the first inequality approximately holds if the function $\theta_{1}^{\alpha}(\cdot)$ is approximately considered as a concave function in $[0,\infty]$ (or when $\frac{X_f-X_0}{|t-T_0|^{1/\alpha}}$ and $\frac{X_f-X_0}{|T_f-t|^{1/\alpha}}$ are large enough).

Thus $\mathcal{P}(X,t)$ reaches the maximal value at $X_0$ when $t\in[T_0,\frac{T_f+T_0}{2}]$, and it reaches the maximal value at $X_f$ when $t\in[\frac{T_f+T_0}{2},T_f]$. At time $\frac{T_f+T_0}{2}$, the maximal value of $\mathcal{P}(X,t)$ is reached at $X_0$ and $X_f$, simultaneously. It thus appears that the transition process jumps at the time instant $\frac{T_f+T_0}{2}$ most probably.

Inspired  by this related observation,    we could choose $t^*=\frac{T_f+T_0}{2}$  in Corollary $\ref{lemmaconvex}$, considering the transition process in time-point-wise sense. This is one  plausible option that leads to the specific most probable path.
\end{myrem}

\begin{corollary}\label{lemmaconvex2}
For symmetric $\alpha$-stable L\'{e}vy motions with $0<\alpha<1$, in $n$-partition path integral representation $(\ref{pdf})$ (that is, the time interval $[T_0,T_f]$ has partition: $T_0=t_0\leq t_1\leq t_2\leq\cdots\leq t_n=T_f$), if the action quantity $\mathcal{S}_n(x_t)$ of a path $x_t$ has more than one non-zero term, then we have
\begin{equation}
\begin{split}
\lim_{n\rightarrow\infty}\frac{S_n(x_t)-n\theta_{\epsilon}^{\alpha}(0)}{S_{n}(u_t)-n\theta_{\epsilon}^{\alpha}(0)}=+\infty.
\end{split}
\end{equation}
Here $u_t\in D_{X_0}^{X_f}$ satisfies
\begin{equation}
\begin{split}
du_t=0,~t\in[T_0,T_f]\setminus\{t^*\},
\end{split}
\end{equation}
where $t^*\in[T_0,T_f]$.
\end{corollary}
\begin{proof}
Suppose that $C_1,C_2,C_3$ are positive constants. We obtain
\begin{equation}\label{order}
\begin{split}
&[\theta^{\alpha}_{\epsilon}(C_1)+\theta^{\alpha}_{\epsilon}(C_2)-2\theta^{\alpha}_{\epsilon}(0)]/(\theta^{\alpha}_{\epsilon}(C_3)-\theta^{\alpha}_{\epsilon}(0))\\
=&[-\ln f_{\alpha}(\frac{C_1}{\epsilon^{1/\alpha}})-\ln f_{\alpha}(\frac{C_2}{\epsilon^{1/\alpha}})+2\ln f_{\alpha}(0)]/[-\ln f_{\alpha}(\frac{C_3}{\epsilon^{1/\alpha}})+\ln f_{\alpha}(0)]\\
=&\ln \{f^2_{\alpha}(0)/[f_{\alpha}(\frac{C_1}{\epsilon^{1/\alpha}})f_{\alpha}(\frac{C_2}{\epsilon^{1/\alpha}})]-f_{\alpha}(0)/(f_{\alpha}(\frac{C_3}{\epsilon^{1/\alpha}}))\}\\
\sim & \ln [C^{-2}f^{2}_{\alpha}(0)(\frac{C_1}{\epsilon^{1/\alpha}})^{1+\alpha}(\frac{C_2}{\epsilon^{1/\alpha}})^{1+\alpha}-C^{-1}f_{\alpha}(0)(\frac{C_3}{\epsilon^{1/\alpha}})^{1+\alpha}]\\
=&\ln [C^{-2}f^{2}_{\alpha}(0)(\frac{C_1C_2}{\epsilon^{2/\alpha}})^{1+\alpha}-C^{-1}f_{\alpha}(0)(\frac{C_3}{\epsilon^{1/\alpha}})^{1+\alpha}]\rightarrow+\infty~(\epsilon\rightarrow0).
\end{split}
\end{equation}

The $\sim$ part and the positive constant $C$ come from the asymptotic behavior of  $f_{\alpha}(\cdot)$ in ($\ref{asymptotic}$). The formula ($\ref{order}$) means that when the jump number of a path is greater, the action quantity of that path has higher order. In $n$-partition path integral representation, if the action quantity $S_n(x_t)=\sum_{i=0}^{n-1}\theta^{\alpha}_{\epsilon}(x_{{t_{i+1}}},x_{t_i})$ of a path $x_t$ has more than one non-zero term, without loss of generality, we assume $\theta^{\alpha}_{\epsilon}(x_{{t_{1}}},x_{t_0})\neq0$ and $\theta^{\alpha}_{\epsilon}(x_{{t_{2}}},x_{t_1})\neq0$. Construct a path $\tilde{x}_t$:
\begin{equation}
 \tilde{x}_t=\begin{cases}
 X_{0},~~T_0\leq t<t_1,\\
 x_{t_1},~~t_1\leq t< t_2,\\
 x_{t_2},~~t_2\leq t<T_f,\\
 X_{f},~~t=T_f.
  \end{cases}
\end{equation}
Applying Corollary $(\ref{lemmaconvex})$ for these two paths in three intervals: $T_0\leq t<t_1$, $t_1\leq t< t_2$, and $t_2\leq t\leq T_f$,  we have
$$\lim_{n\rightarrow\infty}\frac{S_n(x_t)-n\theta_{\epsilon}^{\alpha}(0)}{S_n(\tilde{x}_t)-n\theta_{\epsilon}^{\alpha}(0)}\geq1.$$
According to ($\ref{order}$),
\begin{equation}
\begin{split}
\lim_{n\rightarrow\infty}\frac{S_n(\tilde{x}_t)-n\theta_{\epsilon}^{\alpha}(0)}{S_{n}(u_t)-n\theta_{\epsilon}^{\alpha}(0)}=+\infty,
\end{split}
\end{equation}
where $u_t$ satisfies $du_t=0,~t\in[T_0,T_f]\setminus\{t^*\}$ for any $t^*\in[T_0,T_f]$.
\end{proof}

\subsection{Most probable transition path in  the case of non-zero drift}
\begin{myTheo}\label{theoremconcave}
(Characterization of the most probable transition path with drift term) \\
For system ($\ref{firstequation}$) and system ($\ref{Btequation}$), assume that the transition probability density   exists, and that  the most probable transition path  $X^m_{t}$ exists and satisfies the integrability condition  $|\int_{T_0}^{t}b(X^m_{s})ds|<\infty$ for $t\in[T_0,T_f]$.

(\romannumeral1) For system ($\ref{firstequation}$) with $L_t$   a symmetric $\alpha$-stable L\'{e}vy motion with $(0<\alpha<1)$,   the most probable transition path  $X^m_{t}$ is determined by the following deterministic dynamical system (i.e., an ordinary differential equation),
\begin{equation}
\begin{cases}
dX^m_{t}-b(X^m_{t})dt=0,~t\in [T_0,T_f]\backslash\{t^*\},\\
X^m_{T_0}=X_{0},~X^m_{T_f}=X_{f},\\
\end{cases}
\end{equation}

(\romannumeral2) For system ($\ref{Btequation}$) with Brownian motion,  the most probable transition path $X^m_t$ is determined by the following deterministic dynamical system (i.e., an integral-differential equation),
\begin{equation}\label{browniancasepath}
  \begin{split}
X^m_t-X_{0}-\int_{T_0}^{t}b(X^m_s)ds=\frac{t-T_0}{T_f-T_0}(X_{f}-X_{0}-\int_{T_0}^{T_f}b(X^m_s)ds),
  \end{split}
\end{equation}
{\color{red}if $|X_{f}-X_{0}-\int_{T_0}^{T_f}b(X^m_s)ds|\leq \inf_{U_t\in D_{X_0}^{X_f}}|X_{f}-X_{0}-\int_{T_0}^{T_f}b(U_s)ds|$.}
\end{myTheo}

\begin{proof}
For the system ($\ref{firstequation}$), the corresponding stochastic integral equation is
\begin{equation}\label{variableexchange}
\begin{split}
&X_t=\int_{T_0}^{t}b(X_s)ds+L_t,\\
&Y_t\triangleq X_t-\int_{T_0}^{t}b(X_s)ds=L_t,
\end{split}
\end{equation}
and the differential form is
\begin{equation}
\begin{split}
X_{t_{i+1}}-X_{t_{i}}-[\gamma b(X_{t_{i+1}})+(1-\gamma) b(X_{t_i})]\triangle t=L_{t_{i+1}}-L_{t_i}.  \\
\end{split}
\end{equation}
In It\^{o} interpretation ($\gamma=0$),
\begin{equation}
\begin{split}
X_{t_{i+1}}-X_{t_i}-b(X_{t_i})\triangle t=L_{t_{i+1}}-L_{t_i}.   \\
\end{split}
\end{equation}
The transition probability density function of $X_t$ is
\begin{equation}
\begin{split}
&p(X_{f},T_f|X_{0},T_0)\\
=&\int_{-\infty}^{\infty}\cdots\int_{-\infty}^{\infty}dL_1\cdots dL_{n-1}dL_{n}p_{\epsilon}(L_1|L_{T_0})p_{\epsilon}(L_2|L_1)\cdots p_{\epsilon}(L_{n}|L_{n-1})\delta(x_{t_n}-X_{f})\\
=&\int \mathcal{D}_{n+1} x\mathcal{J}\exp\{-\sum_{i=0}^{n}\theta^{\alpha}_{\epsilon}(x_{{t_{i+1}}}-x_{t_i}-b(x_{t_i})\triangle t)\}\delta(x_{t_n}-X_{f})\\
=&\int \mathcal{D}_{n+1} x\exp\{-\sum_{i=0}^{n}\theta^{\alpha}_{\epsilon}(x_{{t_{i+1}}}-x_{t_i}-b(x_{t_i})\triangle t)\}\delta(x_{t_n}-X_{f})\\
=&\int \mathcal{D}_{n+1} y\exp\{-\sum_{i=0}^{n}\theta^{\alpha}_{\epsilon}(y_{{t_{i+1}}},y_{t_i})\}\delta(x_{t_n}-X_{f}),
\end{split}
\end{equation}
where $\mathcal{J}$ is the Jacobian of the   transformation given by
\begin{equation}
\begin{split}
\mathcal{J}=\det(\frac{\partial L_i}{\partial x_k})=\prod_{i=1}^n(1-\epsilon\gamma\frac{db(x_i)}{dx_i}).
\end{split}
\end{equation}

Assume that the most probable transition path of $X_t$ exist, which is denoted by $X^m_t$.

(\romannumeral1)  For system ($\ref{firstequation}$), in order  to determine the most probable transition path $X^m_t$,  we consider the transition of the process $Y_t$ from $Y_{T_0}=X_0$ to $Y_{T_f}=X_{f}-\int_{T_0}^{T_f}b(X^m_s)ds$. We should notice that the transition process of $Y_t$ is different from the one of $X_t$.  Given the quantities $\{X_0,X_f,T_0,T_f\}$. The diffusion process $X_t$ transfers from initial point $X_0$ at time $T_0$ to terminal point $X_f$ at time $T_f$. That is, all transition paths have the same initial and terminal points. But for process $Y_t$, the transition paths set of $Y_t$ is
\begin{align}
D_{Y}=\{y_t:y_t=x_t-\int_{T_0}^{t}b(x_s)ds,~x_t\in D_{X_0}^{X_f}\}.
\end{align}
Thus the paths in $D_Y$ have the same initial point $X_0$ but their terminal points may be different.

Since $L_t$ is an $\alpha$-stable L\'{e}vy motion with $0<\alpha<1$, by Corollary $\ref{lemmaconvex2}$, the most probable transition path of $Y_t$ (denoted by $Y^m_t$) among $D_{Y}$ is presumably to satisfy
\begin{equation}
\begin{split}
dY^m_t=0,~t\in[T_0,T_f]\setminus\{t^*\},
\end{split}
\end{equation}
where $t^*\in[T_0,T_f]$.

That  is
\begin{equation}
\left\{
\begin{array}{l}
dX^m_t-b(X^m_t)dt=0,\\
X^m_{T_0}=X_{0},~X^m_{T_f}=X_{f},~t\in [T_0,T_f]\backslash\{t^*\}.  \\
\end{array}
\right.
\end{equation}

It means that the most probable transition path has one jump and the jump size is $J_{t^*}=|X_0-\int_{T_0}^{t^*}b(X^m_s)ds-X_f+\int_{T_f}^{t^*}b(X^m_s)ds|$. In the proof of Corollary $\ref{lemmaconvex2}$, if $C_1>C_3>0$ and $C_2=0$, the limit is still  the infinity. This means the order of action quantity depends on the jump size.

So the jump time $t^*$ is supposed to be the one which satisfies
\begin{equation}
\begin{split}
J_{t^*}=\min_{T_0\leq t\leq T_f}J_t.
\end{split}
\end{equation}
As $J_t$ is continuous in $t$, the minimizer $t^*$ exists (although it may not be unique).

In other words, the most probable transition path consists of two components: One component is part of the solution of equation
\begin{equation}\label{path1}
\left\{
\begin{array}{l}
dX(t)-b(X(t))dt=0,~t\in [T_0,T_f],\\
X(T_0)=X_{0},\\
\end{array}
\right.
\end{equation}
and the other component is part of the solution of equation
\begin{equation}\label{path2}
\left\{
\begin{array}{l}
dX(t)-b(X(t))dt=0,~t\in [T_0,T_f],\\
X(T_f)=X_{f}.\\
\end{array}
\right.
\end{equation}
The jump time $t^*$ is the moment at which $J_{t^*}$ is minimum.

(\romannumeral2) For  system ($\ref{Btequation}$), if the initial and terminal points are deterministic, then by Corollary $\ref{lemmaconvex}$, the most probable transition path of a Brownian motion is the one which connects the initial and terminal points directly. Actually the action quantity of this most probable transition path can be computed exactly provided the equal time partition in ($\ref{browniancovex}$). That is,
\begin{equation}
\begin{split}
&n\theta^b_{\epsilon}(\frac{1}{n}|X_{f}-X_{0}|)\\
=&n\frac{(|X_{f}-X_{0}|/n)^2}{2\epsilon}\\
=&n\frac{(|X_{f}-X_{0}|/n)^2}{2\frac{T_f-T_0}{n}}\\
=&\frac{1}{2}\frac{(X_{f}-X_{0})^2}{T_f-T_0}.
\end{split}
\end{equation}

Notice that $|Y_{T_f}-Y_{T_0}|=|X_{f}-X_{0}-\int_{T_0}^{T_f}b(X_s)ds|$ is the distance between the initial and terminal points of $Y_t$. Then the most probable transition path of $Y_t$  (denoted by $Y^m_t$)  can  be obtained exactly,
\begin{equation}\label{bre}
\begin{split}
Y^m_t&= X^m_t-X_{0}-\int_{T_0}^{t}b(X^m_s)ds\\
&=Y_{0}+(Y_{f}-Y_{0})\frac{t-T_0}{T_f-T_0}\\
&=\frac{t-T_0}{T_f-T_0}(X_{f}-X_{0}-\int_{T_0}^{T_f}b(X^m_s)ds).
\end{split}
\end{equation}
This means if there is a transition path $X^m_t\in D_{X_0}^{X_f}$ such that the formula $(\ref{bre})$ holds, then the path $X^m_t$ is the most probable one among the paths whose terminal point is $X_f-\int_{T_0}^{T_f}b(X^m_s)ds$. So if $|X_{f}-X_{0}-\int_{T_0}^{T_f}b(X^m_s)ds|\leq \inf_{U_t\in D_{X_0}^{X_f}}|X_{f}-X_{0}-\int_{T_0}^{T_f}b(U_s)ds|$, then $X^m_t$ will be the most probable transition path of $X_t$.

This completes the proof of this theorem.
\end{proof}

\begin{myrem}
For a symmetric $\alpha$-stable L\'{e}vy motion with $1\leq\alpha<2$, the equalities $(\ref{control})$ do not hold.  Define the path space $\mbox{DMS}_{X_{0}}^{X_{f}}\triangleq\{x_t|~x_t=\sum_{i=0}^{n-1}a_i \mathbb{I}_{[t_i,t_{i+1})}+a_n \mathbb{I}_{\{T_f\}}, T_0=t_0\leq t_1\leq \cdots\leq t_n=T_f, X_{0}=a_0\leq a_1\leq \cdots\leq a_n=X_{f}\}$, where $\mathbb{I}_{B}$ is the characteristic function of set $B\subset \mathbb{R}$.

If we search for the most probable transition path within the simple path space $DMS_{X_{0}}^{X_{f}}$ for the symmetric $\alpha$-stable L\'{e}vy motion with $1\leq\alpha<2$, the similar results  of Lemma $\ref{lemmaconvex}$ and Theorem $\ref{theoremconcave}$ hold. This is because  for every  fixed simple path   in $DMS_{X_{0}}^{X_{f}}$, the non-zero $\lambda_j$ of $\{\lambda_j\}_{j=0}^{n-1}$ are bounded below. Thus the inequality ($\ref{concavity}$) holds for every  non-negative $\lambda_j$,   for $n$ large enough.
\end{myrem}

\begin{myrem}
For simplicity, we call the solution of ($\ref{path1}$) the \textbf{initial transition path}, and call the solution of ($\ref{path2}$) the \textbf{final transition path}. Consequently, the most probable transition path starts from the initial transition path and jumps to the final transition path at the time that initial transition path and final transition path are closed to each other.
\end{myrem}

\begin{myrem}
For  the case with a symmetric $\alpha$-stable L\'{e}vy motion  ($0<\alpha<1$), we are interested in the transitions     between metastable states of the stochastic dynamical system $(\ref{firstequation})$. That is, the initial and terminal points $X_0$ and $X_f$ are stable points of the corresponding undisturbed system of $\ref{firstequation}$. In this case, the initial transition path is $X_t=X_0,~t\in[T_0,T_f]$ and the final transition path is $X_t=X_f,~t\in[T_0,T_f]$ by Theorem $\ref{theoremconcave}$ (\romannumeral1).

Thus the process $Y_t=X_t+\int_{T_0}^{T_f}b(X_s)ds$ is a symmetric $\alpha$-stable L\'{e}vy motion which transits from $X_0$ to $X_f$ most probably. The most probable transition path of $Y_t$ provided the initial and terminal points $X_0$ and $X_f$ (as discussed in Remark $\ref{remark1}$) is
\begin{equation}
\begin{cases}
Y^m_{t}=X_0,~t\in [T_0,\frac{T_f+T_0}{2}),\\
Y^m_{t}=X_f,~t\in [\frac{T_f+T_0}{2},T_f].\\
\end{cases}
\end{equation}
Thus in Theorem $\ref{theoremconcave}$ (\romannumeral1), if the initial and terminal points $X_0$ and $X_f$ are metastable points of the system, the time $\frac{T_f+T_0}{2}$ can also be considered as the most probable jump time $t^*$.
\end{myrem}

\begin{myrem}\label{remark5}
For  the case with Brownian motion, the condition $|X_{f}-X_{0}-\int_{T_0}^{T_f}b(X^m_s)ds|\leq \inf_{U_t\in D_{X_0}^{X_f}}|X_{f}-X_{0}-\int_{T_0}^{T_f}b(U_s)ds|$ is not easy to verify. The action functional in It\^{o} sense is $\mathcal{L}(x,\dot{x})=\frac{1}{2}(\dot{x}-b(x))^2$. If we assume the most probable path and the function $b(x)$ are smooth enough, then the Euler-Lagrange equation is
\begin{equation}
\begin{split}
&~~\frac{\partial}{\partial t}\frac{\partial \mathcal{L}}{\partial \dot{x}}=\frac{\partial \mathcal{L}}{\partial x}\\
\Rightarrow&~~\ddot{x}-b'(x)\dot{x}=-(\dot{x}-b(x))b'(x)\\
\Rightarrow&~~\ddot{x}-b'(x)b(x)=0.\\
\end{split}
\end{equation}
And in our method,
\begin{equation}
  \begin{split}
&~~X^m_t-X_{0}-\int_{T_0}^{t}b(X^m_s)ds=\frac{t-T_0}{T_f-T_0}(X_{f}-X_{0}-\int_{T_0}^{T_f}b(X^m_s)ds)\\
\Rightarrow&~~\dot{X}^m_t-b(X^m_t)=\frac{1}{T_f-T_0}(X_{f}-X_{0}-\int_{T_0}^{T_f}b(X^m_s)ds)\\
\Rightarrow&~~\ddot{X}^m_t-b'(X^m_t)\dot{X}^m_t=0.\\
  \end{split}
\end{equation}
So in   general,   the path of ($\ref{browniancasepath}$)  does not coincide with direct
  minimizer of the Onsager-Machlup's functional. Our method is restricted here because the condition $|X_{f}-X_{0}-\int_{T_0}^{T_f}b(X^m_s)ds|\leq \inf_{U_t\in D_{X_0}^{X_f}}|X_{f}-X_{0}-\int_{T_0}^{T_f}b(U_s)ds|$ is not satisfied. But if the drift term $b(x)$ is independent of $x$, the integration $\int_{T_0}^{T_f}b(U_s)ds$ is a constant for any path $U_t\in D_{X_0}^{X_f}$, and the two methods provide the same most probable path.
\end{myrem}

\subsection{Existence, uniqueness and numerical simulation for the most probable transition path.}
\subsubsection{Non-Gaussian noise: System with a symmetric $\alpha$-stable L\'{e}vy motion with $0<\alpha<1$.}

\par The  most probable transition path is determined by  two ``initial" value problems  of a deterministic ordinary differential equation (ODE)
\begin{equation}\label{ode1}
\begin{cases}
dX^m_t-b(X^m_t)dt=0,~t\in [T_0,T_f]\backslash\{t^*\},\\
X^m_{T_0}=X_{0},~X^m_{T_f}=X_{f},~t^*\in[T_0,T_f].\\
\end{cases}
\end{equation}
The first initial value problem  solves this ODE with  $X^m_{T_0}=X_{0}$, and the second problem solves this ODE backward in time with terminal value condition  $X^m_{T_f}=X_{f}$.
The existence and uniqueness of these solutions are ensured  by the local Lipschitz continuity of the drift term $b(x)$.

Given a time partition $\{T_0=t_0<t_1<\cdots<t_n=T_f\}$,  we simulate the most probable transition path as follows:   forward Euler scheme
\begin{align}\label{initialtransition}
x_{t_{i+1}}-x_{t_{i}}=b(x_{t_{i}})\Delta t,~T_0\leq t_i<t_{i+1}<t^*,~x_{T_0}=X_{0},
\end{align}
and
\begin{align}\label{finaltransition}
x_{t_{i+1}}-x_{t_{i}}=b(x_{t_{i}})\Delta t,~t^*\leq t_i<t_{i+1}\leq T_f,~x_{T_f}=X_{f}.
\end{align}

The initial transition path simulated by ($\ref{initialtransition}$) can be computed easily. Computing the final transition path   by ($\ref{finaltransition}$) is a little complicated. Since the differential equation $dX_t=b(X_t)dt$ has an unique solution, provided it  passes through a given point at given time.   For  $b(X_t)=0$,   the final transition path is  $X_t=X_f,~t^*\leq t\leq T_f$.

The difference scheme we used in ($\ref{initialtransition}$) and ($\ref{finaltransition}$) is consistent with the It\^{o} interpretation. The differences between It\^{o} interpretation and other stochastic interpretations could be found in $\cite{Arenas2012,Moreno2015}$ and references therein.
\subsubsection{Gaussian noise: System with a Brownian motion.}

\par In this case, the most probable transition path is determined by a deterministic integral-differential equation
\begin{equation}\label{ode2}
  \begin{split}
X^m_t-X_{0}-\int_{T_0}^{t}b(X^m_s)ds=\frac{t-T_0}{T_f-T_0}(X_{f}-X_{0}-\int_{T_0}^{T_f}b(X^m_s)ds).
  \end{split}
\end{equation}
with a constraint $|X_{f}-X_{0}-\int_{T_0}^{T_f}b(X^m_s)ds|\leq \inf_{U_t\in D_{X_0}^{X_f}}|X_{f}-X_{0}-\int_{T_0}^{T_f}b(U_s)ds|$ which has been discussed in Remark $\ref{remark5}$.

\section{Higher Dimensional Cases}
In this section, we discuss the higher dimensional cases.
We consider an SDE system with non-Gaussian noise
{\small\begin{equation}\label{kdimension1}
\left\{
\begin{array}{l}
dX_{1,t}=b_1(X_{1,t},X_{2,t},\cdots,X_{k,t})dt+dL_{1,t},~X_{1,T_0}=X_{1,0},\\
dX_{2,t}=b_2(X_{1,t},X_{2,t},\cdots,X_{k,t})dt+dL_{2,t},~X_{2,T_0}=X_{2,0},\\
~~~~\vdots~~~~~~~~~~~~~\vdots~~~~~~~~~~~~~~~~~~~~~~~~~~~~~~~~~~~~~~~~~~~~~~~\vdots\\
dX_{k,t}=b_k(X_{1,t},X_{2,t},\cdots,X_{k,t})dt+dL_{k,t},~X_{k,T_0}=X_{k,0},\\
\end{array}\right.
\triangleq dX_t=b(X_t)dt+dL_t,~X_{T_0}=X_0,
\end{equation}}
where $L_{i,t}$ are symmetric $\alpha$-stable L\'{e}vy noises ($0<\alpha<1$)  and $\{L_{i,t},\cdots,L_{j,t}\}$ are independent, and an SDE system with Gaussian noise
{\small\begin{equation}\label{kdimension2}
\left\{
\begin{array}{l}
dX_{1,t}=b_1(X_{1,t},X_{2,t},\cdots,X_{k,t})dt+dB_{1,t},~X_{1,T_0}=X_{1,0},\\
dX_{2,t}=b_2(X_{1,t},X_{2,t},\cdots,X_{k,t})dt+dB_{2,t},~X_{2,T_0}=X_{2,0},\\
~~~~\vdots~~~~~~~~~~~~~\vdots~~~~~~~~~~~~~~~~~~~~~~~~~~~~~~~~~~~~~~~~~~~~~~~\vdots\\
dX_{k,t}=b_k(X_{1,t},X_{2,t},\cdots,X_{k,t})dt+dB_{k,t},~X_{k,T_0}=X_{k,0},\\
\end{array}\right.
\triangleq dX_t=b(X_t)dt+dB_t,~X_{T_0}=X_0,
\end{equation}}
where $B_{i,t}$ are Brownian motions, and $\{B_{1,t},\cdots,B_{k,t}\}$ are independent.

It was known  $\cite{Ash1999}$ that the random variables $X_1,\cdots,X_k$ are independent if and only if  $f(x_1,\cdots,x_k)=f_1(x_1)\cdots f_k(x_k)$ for all $(x_1,\cdots,x_n)$ except possibly for a Borel subset of $\mathbb{R}^k$ with Lebesgue measure zero. Here $f$ is the probability density of $(X_1,\cdots,X_k)$ and $f_i$ is the probability density of $X_i$ ($i=1,\cdots,k$). Hence by the independence of the noises, the probability density function denoted by $f^k_{\alpha}(x)$ of k-dimensional $\alpha$-stable L\'{e}vy variable $x=(x_{1},x_{2},\cdots,x_{k})\in\mathbb{R}^k$ is
\begin{equation}
\begin{split}
f^k_{\alpha}(x)&=f_{\alpha}(x_1)f_{\alpha}(x_2)\cdots f_{\alpha}(x_k)\\
&=\exp\{-[-\sum_{i=1}^{k}\ln(f_{\alpha}(x_i))]\}\\
&=\exp\{-\sum_{i=1}^{k}\theta_1^{\alpha}(x_i)\}.\\
\end{split}
\end{equation}

Recall in Theorem $\ref{mtheorem}$, we proved that the most probable transition path is supposed to be monotonic  with respect to time $t$. In higher dimensional cases, the transition probability density function has the similar form of ($\ref{pdf}$). It implies that every component of the most probable transition path are monotonic with respect to time $t$.

For the system ($\ref{kdimension1}$) with zero drift term, the time partition $\{T_0=t_0<t_1<\cdots<t_n=T_f,t_{i+1}-t_i=\frac{T_f-T_0}{n}\}$, and for a path $x_t=(x_{1,t},x_{2,t},\cdots,x_{k,t})$ connecting  $(T_0,X_{0})$ and $(T_f,X_{f})$ monotonically, we denote $x_{i,t_{j+1}}-x_{i,t_{j}}=\lambda^i_j(X_{i,f}-X_{i,0})$.  Thus
\begin{equation}
\begin{split}
\mathcal{S}_n(x_t)&=\sum_{i=1}^{k}\sum_{j=0}^{n-1}\theta^{\alpha}_{\epsilon}(x_{i,t_{j+1}}-x_{i,t_j})-kn\theta^{\alpha}_{\epsilon}(0)\\
&=\sum_{i=1}^{k}\sum_{j=0}^{n-1}[\theta^{\alpha}_{1}(\lambda^i_jn^{1/\alpha}\frac{X_{i,f}-X_{i,0}}{(T_f-T_0)^{1/\alpha}})-\theta^{\alpha}_{\epsilon}(0)]\\
&\geq\sum_{i=1}^{k}\sum_{\lambda^i_jn^{1/\alpha}\frac{X_{i,f}-X_{i,0}}{(T_f-T_0)^{1/\alpha}}\geq r_i}[\theta^{\alpha}_{1}(\lambda^i_jn^{1/\alpha}\frac{X_{i,f}-X_{i,0}}{(T_f-T_0)^{1/\alpha}})-\theta^{\alpha}_{\epsilon}(0)]\\
&\geq\sum_{i=1}^{k}\sum_{\lambda^i_jn^{1/\alpha}\frac{X_{i,f}-X_{i,0}}{(T_f-T_0)^{1/\alpha}}\geq r_i}\lambda^i_j[\theta^{\alpha}_{1}(n^{1/\alpha}\frac{X_{i,f}-X_{i,0}}{(T_f-T_0)^{1/\alpha}})-\theta^{\alpha}_{\epsilon}(0)]\\
&=\sum_{i=1}^{k}[\theta^{\alpha}_{1}(n^{1/\alpha}\frac{X_{i,f}-X_{i,0}}{(T_f-T_0)^{1/\alpha}})-\theta^{\alpha}_{\epsilon}(0)]\sum_{\lambda^i_jn^{1/\alpha}\frac{X_{i,f}-X_{i,0}}{(T_f-T_0)^{1/\alpha}}\geq r_i}\lambda^i_j\\
&\geq\sum_{i=1}^{k}[\theta^{\alpha}_{1}(n^{1/\alpha}\frac{X_{i,f}-X_{i,0}}{(T_f-T_0)^{1/\alpha}})-\theta^{\alpha}_{\epsilon}(0)](n\rightarrow\infty).
\end{split}
\end{equation}
This implies   the results of Corollary $\ref{lemmaconvex}$ and Theorem $\ref{theoremconcave}$ in higher dimensional cases and they are similar to 1-dimensional case.

For the system ($\ref{kdimension2}$) with zero drift term,
\begin{equation}
\begin{split}
\mathcal{S}_n(x_t)&=\sum_{i=1}^{k}\sum_{j=0}^{n-1}\theta^b_{\epsilon}(x_{i,t_{j+1}}-x_{i,t_j})\\
&=\sum_{i=1}^{k}\sum_{j=0}^{n-1}\theta^b_{\epsilon}(|x_{i,{t_{j+1}}}-x_{i,t_j}|)\\
&\geq \sum_{i=1}^{k}n\theta^b_{\epsilon}(\sum_{j=0}^{n-1}\frac{1}{n}|x_{i,{t_{j+1}}}-x_{i,t_j}|)\\
&=\sum_{i=1}^{k}n\theta^b_{\epsilon}(\frac{1}{n}|X_{i,f}-X_{i,0}|).
\end{split}
\end{equation}
This implies   the results of Corollary $\ref{lemmaconvex}$ and Theorem $\ref{theoremconcave}$ in higher dimensional cases  and they are similar to 1-dimensional case.

\section{Examples}

Let us consider several examples in order to illustrate our results.
\begin{myexample}
Ornstein-Uhlenbeck process
\end{myexample}
\noindent Consider a linear scalar SDE:
$$dX_t=rdt+dL_t,~T_0\leq t\leq T_f,~X_{T_0}=X_{0},~X_{T_f}=X_{f},$$
with $r$ a constant.  Let $Y_t=X_t-r(t-T_0)$. Then by $It\hat{o}$ formula
\begin{align*}
dY_t&=-rdt+dX_t=dL_t.
\end{align*}

When $L_t$ is a symmetric $\alpha$-stable L\'{e}vy motion with $0<\alpha<1$, the most probable transition path of $Y_t$ is
\begin{equation*}
Y^m_t=\begin{cases}
 X_f-r(T_f-T_0),~~T_f\geq t\geq t^*~( t>t^*),\\
 X_0,~~T_0\leq t<t^*~(t\leq t^*),\\
  \end{cases}
\end{equation*}
for every $t^*$ satisfying  $T_0\leq t^*\leq T_f$. Thus the most probable transition path for $X_t$ is
\begin{equation*}
X^m_t=\begin{cases}
 X_{f}-r(T_f-t),~~t\geq t^*~(t>t^*),\\
 X_{0}+r(t-T_0),~~t<t^*~(t\leq t^*),\\
  \end{cases}
\end{equation*}
where $t^*$ can be chosen arbitrarily in $[T_0,T_f]$.

When $L_t$ is replaced by a Brownian motion,
$$dX_t=rdt+dB_t,~T_0\leq t\leq T_f,~X_{T_0}=X_{0},~X_{T_f}=X_{f},$$
the most probable transition path of $X_t$ is (by Theorem $\ref{theoremconcave}$),
\begin{align*}
X^m_t-r(t-T_0)&=X_{0}+\frac{t-T_0}{T_f-T_0}(X_{f}-X_{0}-r(T_f-T_0))\\
\Leftrightarrow X^m_t &=X_{0}+\frac{t-T_0}{T_f-T_0}(X_{f}-X_{0}).
\end{align*}
So in this linear system with drift and Gaussian noise, the  most probable transition path is also a line segment.

Figure $\ref{lineargraph}$ shows the most probable transition paths of this example.
\begin{figure}
  \centering
  \includegraphics[width=8cm,height=6cm]{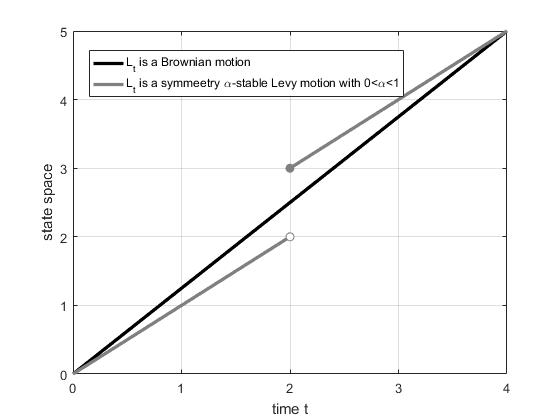}\\
  \caption{The most probable transition paths of Ornstein-Uhlenbeck process in Example 1 with $T_0=0,~T_f=4,~r=1,~X_{0}=0,~X_{f}=5$. We choose $t^*=2$ and the  `right-continuous with left limit' version of the most probable transition path.}\label{lineargraph}
\end{figure}

\begin{myexample}
Geometric Brownian motion
\end{myexample}
\noindent Consider a linear scalar SDE with multiplicative noise
$$dX_t=rX_tdt+\eta X_tdB_t,$$
where $r$ and $\eta$ are real constants, and $X_t>0$, $a.s.$.  Setting  $Y_t=\ln X_t-(r-\frac{1}{2}\eta^2)t$ and  applying It\^{o} formula, we obtain
\begin{align*}
dY_t&=d\ln(X_t)-(r-\frac{1}{2}\eta^2)dt\\
&=\frac{1}{X_t}dX_t+\frac{1}{2}(-\frac{1}{X_t^2})(dX_t)^2-(r-\frac{1}{2}\eta^2)dt\\
&=\frac{dX_t}{X_t}-\frac{1}{2}\eta^2dt-(r-\frac{1}{2}\eta^2)dt\\
&=\eta dB_t.
\end{align*}

By Corollary $\ref{lemmaconvex}$,
\begin{align*}
Y^m_t=Y_0+\frac{t-T_0}{T_f-T_0}(Y_f-Y_0).
\end{align*}
Thus
\begin{align*}
X^m_t=\exp\{(r-\frac{1}{2}\eta^2)t+\ln X_{0}+\frac{t-T_0}{T_f-T_0}[\ln X_{f}-(r-\frac{1}{2}\eta^2)T_f-\ln X_{0}]\}.
\end{align*}

Figure $\ref{geograph}$ shows the most probable transition path of this example.

\begin{myexample}
Geometric L\'{e}vy process
\end{myexample}
\noindent Consider the stochastic differential equation
$$dX_t=X_{t}[\zeta dt+\beta dB(t)+\int_{\mathbb{R}}\gamma(t,z)\bar{N}(dt,dz)],$$
where $\zeta,~\beta$ are constants and $\gamma(t,z)\geq-1$, and
\begin{equation*}
  \bar{N}(dt,dz)=\begin{cases}
 N(dt,dz)-\nu(dz)dt,~if~|z|<r,\\
 N(dt,dz),~if~|z|\geq r,\\
  \end{cases}
\end{equation*}
$r\in\mathbb{R}+$.  For simplicity we set $\beta=0,~ \gamma(t,z)=e^z-1$. Now define $Y_t=\ln X_t$. By It\^{o} formula, we have
\begin{align*}
dY_t&=\zeta dt+\int_{|z|<r}\{\ln(1+e^z-1)-(e^z-1)\}\nu(dz)dt+\int_{\mathbb{R}}\ln(1+e^z-1)\bar{N}(dt,dz)\\
&=\zeta dt+\int_{|z|<r}\{z-(e^z-1)\}\nu(dz)dt+\int_{\mathbb{R}}z\bar{N}(dt,dz).
\end{align*}

Let $U_t=Y_t-(\zeta +\int_{|z|<r}\{z-(e^z-1)\}\nu(dz))t$ be a symmetric $\alpha$-stable L\'{e}vy process, i.e., $dU_t=\int_{r}z\bar{N}(dt,dz)=dL^{\alpha}_t$.

Let the jump measure $\nu(dz)$ be the jump measure of an $\alpha$-stable L\'{e}vy process,  that is, $\nu(dz)=c_{\alpha}\frac{dz}{|z|^{1+\alpha}}$ with
$$c_{\alpha}=\frac{\alpha}{2^{1-\alpha}\sqrt{\pi}}\frac{\Gamma(\frac{1+\alpha}{2})}{\Gamma(1-\frac{\alpha}{2})}.$$

For $\alpha$ with  $0<\alpha<1$ and $~r=1$,
$$\int_{|z|<1}\{z-(e^z-1)\}\nu(dz)<\infty. $$
By Corollary $\ref{lemmaconvex}$,  the most probable transition path of $U_t$ is
\begin{equation*}
  U^m_t=\begin{cases}
 U_0,~T_0\leq t<t^*~(T_0\leq t\leq t^*),\\
 U_f,~t^*\leq t\leq T_f~(t^*< t\leq T_f),
  \end{cases}
\end{equation*}
where $t^* \in[T_0,T_f]$. Thus
\begin{equation*}
  X^m_t=\begin{cases}
 \exp(\ln X_{0}+(t-T_0)[\zeta+\int_{|z|<1}\{z-(e^z-1)\}\nu(dz)]),~T_0\leq t<t^*~(T_0\leq t\leq t^*),\\
 \exp(\ln X_{f}+(t-T_f)[\zeta+\int_{|z|<1}\{z-(e^z-1)\}\nu(dz)]),~t^*\leq t\leq T_f~(t^*< t\leq T_f).\\
  \end{cases}
\end{equation*}

Figure $\ref{geograph}$ also shows the most probable transition path of this example.
\begin{figure}
 \centering
 \subfigure[]{
 \includegraphics[width=7cm,height=6cm]{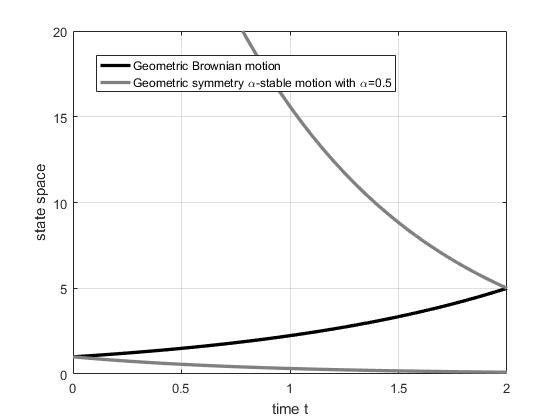}
 }
 \subfigure[]{
 \includegraphics[width=7cm,height=6cm]{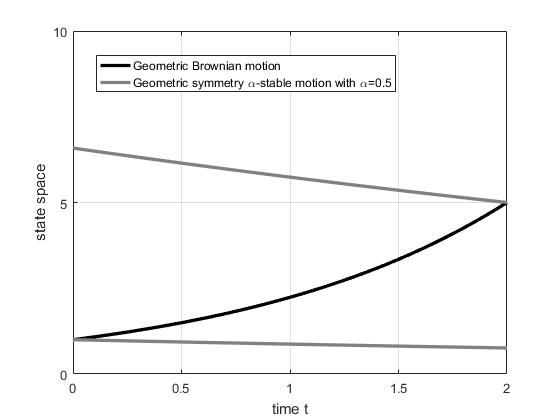}
 }
 \subfigure[]{
 \includegraphics[width=7cm,height=6cm]{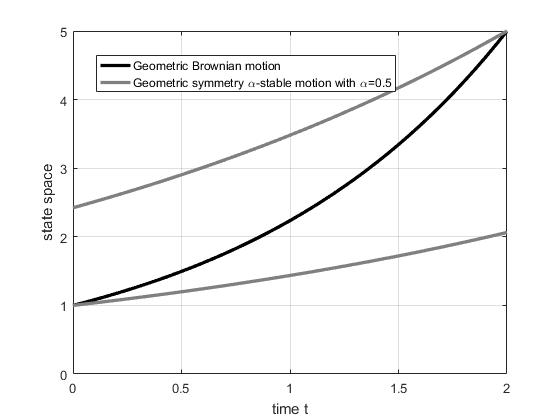}
 }
 \subfigure[]{
 \includegraphics[width=7cm,height=6cm]{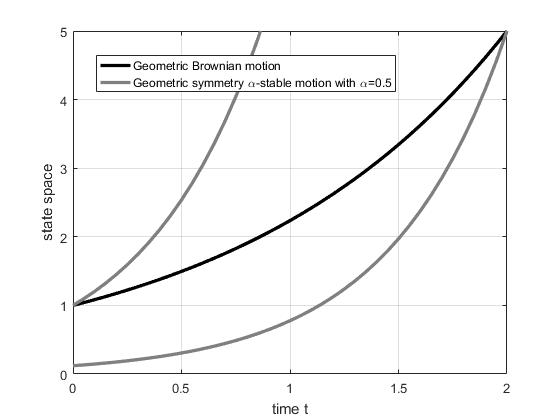}
 }
 \caption{The most probable transition paths of Geometric Brownian motion (black solid line) and the initial and final transition path of geometry symmetric $\alpha$-stable L\'{e}vy motion with $\alpha=0.5$ (grey solid lines, the most probable path starts at initial transition path and jumps to final transition path at time $t^*=1$),   for $T_0=0,~T_f=2,~X_{0}=1,~X_{2}=5,~r=1,~\eta=1,~\beta=0$:   $(a)~\zeta=-1;(b)~\zeta=0;(c)~\zeta=0.5;(d)~\zeta=2.$}\label{geograph}
 \end{figure}

\begin{myexample}
One-dimensional Nonlinear SDE:  Stochastic double-well system
\end{myexample}
\noindent Consider the stochastic double-well system
\begin{align*}
dX_t=(X_t-X_t^3)dt+dL_t,
\end{align*}
where  $L_t$  is  a symmetric $\alpha$-stable L\'{e}vy motion with $0<\alpha<1$.

The corresponding undisturbed system has three equilibrium points: -1, 0, 1 ( -1 and 1 are stable equilibrium points,  0 is an unstable equilibrium point).

By Theorem $\ref{theoremconcave}$, the most probable transition path  of this system is described by the following deterministic differential equation:
\begin{equation*}
\begin{cases}
dX^m_{t}-(X^m_t-(X^m_t)^3)dt=0,~t\in [T_0,T_f]\backslash\{t^*\},\\
X^m_{T_0}=X_{0},~X^m_{T_f}=X_{f},~t^*\in[T_0,T_f]. \\
\end{cases}
\end{equation*}

We compute some solution curves of above system as  shown in Figure $\ref{double}$. The most probable  transition path consists of the solutions of this deterministic system. For instance, if we consider: $X_0=1$, $X_f=-1$, $T_0=0$, $T_f=4$, then the most probable transition path consists of the parts of two straight lines in Figure $\ref{double}$.
\begin{figure}
  \centering
  \includegraphics[width=0.6\textwidth]{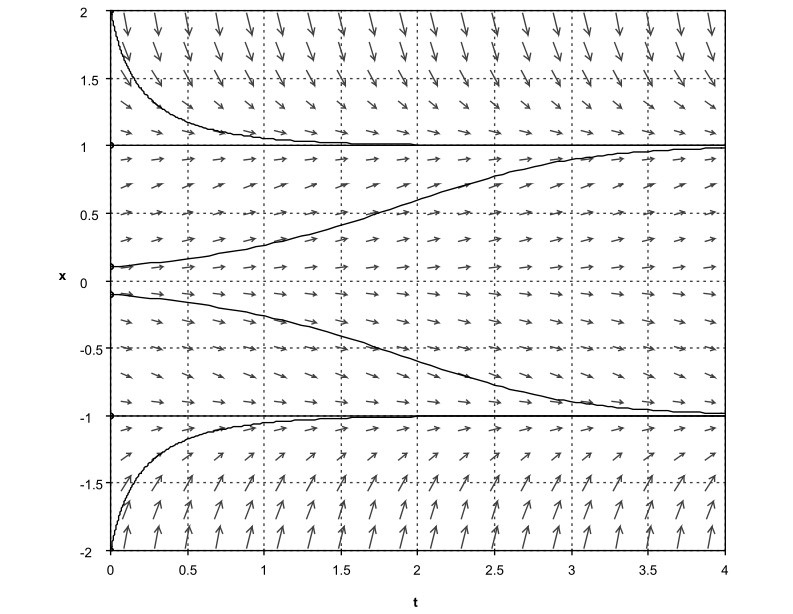}\\
  \caption{Solution curves for $\dot{x}=x-x^3$.}\label{double}
\end{figure}

\begin{myexample}
Two-Dimensional Nonlinear SDE: The Maier-Stein model
\end{myexample}
\noindent Consider the following SDEs:
\begin{equation*}
\begin{cases}
du=(u-u^3-\beta uv^2)dt+dL_{1,t},\\
dv=-(1+u^2)vdt+dL_{2,t},
\end{cases}
\end{equation*}
By Theorem $\ref{theoremconcave}$,
the most probable transition path  $(u^m_t,v^m_t)$ of this  system is described by the following deterministic differential equations:
\begin{equation*}
\begin{cases}
du^m_{t}=(u^m_{t}-(u^m_{t})^3-\beta u^m_{t}(v^m_{t})^2)dt,~t\in [T_0,T_f]\backslash\{t^*\},\\
dv^m_{t}=-(1+(u^m_{t})^2)v^m_{t}dt, ~t\in [T_0,T_f]\backslash\{t^*\},\\
u^m_{T_0}=u_{0},~u^m_{T_f}=u_{f},~v^m_{T_0}=v_{0},~v^m_{T_f}=v_{f},~t^*\in[T_0,T_f].  \\
\end{cases}
\end{equation*}

Figure $\ref{2d}$ shows the phase portrait of this deterministic system.   There are three equilibrium points:  $(-1,0), (0,0), (1,0)$. In Figure $\ref{2d}$ we show several  orbits in black lines.   The most probable transition path can be found by the phase portrait with given initial and terminal conditions.
\begin{figure}
  \centering
  \includegraphics[width=0.6\textwidth]{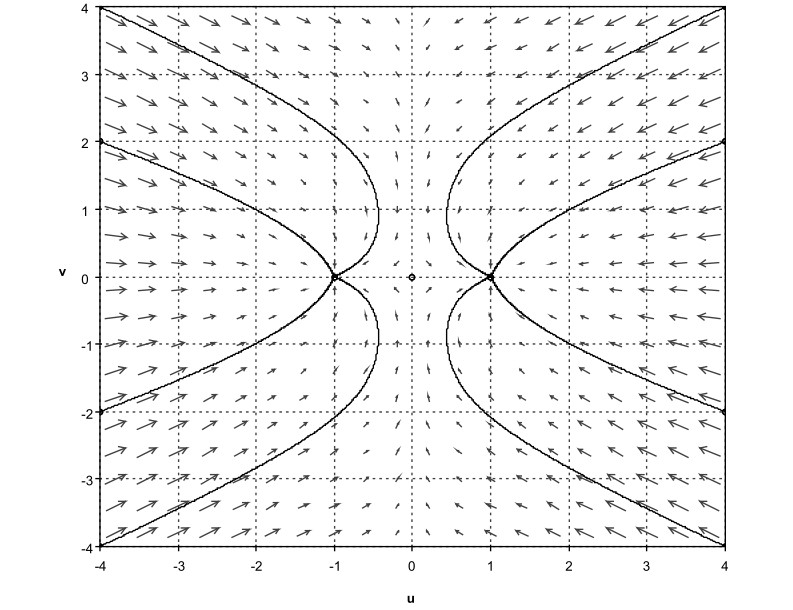}\\
  \caption{Phase portrait for the Maier-Stein model.}\label{2d}
\end{figure}


\section*{Acknowledgements}
The authors would like to thank Professor Xu Sun, Dr Qiao Huang, Dr Yayun Zheng, Dr Wei Wei, Dr Ao Zhang,  and Dr Jianyu Hu for helpful discussions. This work was partly supported by the NSF grant 1620449, and NSFC
grants 11531006 and 11771449.


\end{document}